\documentclass[11pt, a4paper]{article}

\usepackage{amsmath, amsthm, amsfonts, amssymb}

\usepackage{setspace}
\usepackage{fullpage}
\usepackage{enumitem}

\usepackage{floatpag}
\usepackage[dvipsnames]{xcolor}

\usepackage[initials]{amsrefs} 
\usepackage{array}
\usepackage{float}
\usepackage{caption}
\usepackage{subcaption} % for subfigures

\usepackage{hyperref}
\hypersetup{colorlinks=true,
	citecolor=blue,
	filecolor=blue,
	linkcolor=blue,
	urlcolor=blue
}
\usepackage{cleveref}

\usepackage{comment}

\newtheorem{lemma}{Lemma}
\newtheorem{theorem}{Theorem} 
 
\newtheorem{definition}{Definition}

\newcommand{\floor}[1]{\left\lfloor{#1}\right\rfloor}
\newcommand{\ceil}[1]{\left\lceil{#1}\right\rceil}
\renewcommand{\quote}[1]{\textquotedblleft #1\textquotedblright}

\begin{document}
	\title{Absolutely avoidable order-size pairs for induced subgraphs}
	\date{\vspace{-5ex}}
	\author{
		Maria Axenovich
		\thanks{
			Karlsruhe Institute of Technology, Karlsruhe, Germany;
			email:
			\mbox{\texttt{maria.aksenovich@kit.edu}}.
		}     
		\and
		Lea Weber
		\thanks{
			Karlsruhe Institute of Technology, Karlsruhe, Germany;
			email: \mbox{\texttt{lea.weber@kit.edu}}.
		}
	}
	
	\maketitle
	%\linespread{1.5}
	
	\begin{abstract}
		\setlength{\parskip}{\medskipamount}
		\setlength{\parindent}{0pt}
		\noindent
	
	We call a pair $(m,f)$ of integers, $m\geq 1$, $0\leq f \leq \binom{m}{2}$, \emph{absolutely avoidable} if there is $n_0$ such that for any pair of integers $(n,e)$ with $n>n_0$ and  $0\leq e\leq \binom{n}{2}$ there is a graph on $n$ vertices and $e$ edges that contains no induced subgraph on $m$ vertices and $f$ edges. Some pairs are clearly not absolutely avoidable, for example $(m,0)$ is not absolutely avoidable since any sufficiently sparse graph on at least $m$ vertices contains independent sets on $m$ vertices. Here we show that there are infinitely many absolutely avoidable pairs. We give a specific infinite set $M$ such that for any $m\in M$, the pair $(m, \binom{m}{2}/2)$ is absolutely avoidable. 
	In addition,  among other results, we show that for any monotone integer function $q(m)$,  $|q(m)|=O(m)$,  there are infinitely many values of $m$ such that the pair $(m, \binom{m}{2}/2 +q(m))$ is absolutely avoidable. 
	\end{abstract}

	\section{Introduction}
One of the central topics of graph theory deals with properties of classes of graphs that contain no subgraph isomorphic to some given fixed graph,  see for example Bollob\'as \cite{B}.
Similarly, graphs with forbidden induced subgraphs have been investigated from several different angles -- enumerative, structural, algorithmic, and more.  \\

Erd\H{o}s, Füredi, Rothschild and S\'os  \cite{EFRS} initiated  a study of a seemingly simpler class of graphs that do not forbid a specific induced subgraph, but rather forbid any induced subgraph on a given number $m$ of vertices and number $f$ of edges.  Following their notation  we say  a graph $G$ \emph{arrows} a pair of non-negative integers $(m,f)$ and write $G \to (m,f)$ if $G$ has an induced subgraph on $m$ vertices and $f$ edges.  We say that  a pair $(n,e)$  of non-negative integers \emph{arrows} the pair $(m,f)$,  and write $(n,e) \to (m,f)$, if  for any graph $G$ on $n$ vertices and $e$ edges, $G \to (m,f)$.  \\
	 As an example, if $t_{m-1}(n)$ denotes the number of edges in the balanced complete $(m-1)$-partite graph on $n$ vertices, then by Tur\'an's theorem \cite{T} we know that any graph on $n$ vertices with more than $t_{m-1}(n)$ edges contains $K_{m}$,  a complete subgraph on $m$ vertices. Equivalently stated, we have $(n,e) \to (m,\binom m2)$  if $e > t_{m-1}(n)$. \\
	 
	 For a fixed pair $(m,f)$ let $$S_n(m,f) = \{e : (n,e) \to (m,f)\}\quad \mbox{ and }\quad \sigma(m,f) =\left. \limsup\limits_{n \to \infty} {|S_n(m,f)|} /{\binom n2}\right.. $$
	 In \cite{EFRS} the authors considered $\sigma(m,f)$ for different choices of $(m,f)$. One of their main results is 
	 \begin{theorem}\cite{EFRS}
	 	If $(m,f) \not\in \{(2,0), (2,1), (4,3), (5,4), (5,6)\}$, then $\sigma(m,f) \le \frac23$; otherwise $\sigma(m,f) =1$.
	 \end{theorem}
 	He, Ma, and Zhao  \cite{HMZ} improved the upper bound $2/3$ to $ 1/2$ and showed that there are infinitely many pairs for which the equality $\sigma(m,f) = \frac12$ holds. \\
 	
 	In \cite{EFRS} the authors also gave a construction demonstrating that \quote{most of the} $\sigma(m,f)$ are $0$, by showing that for large $n$ almost all pairs $(n,e)$ can be realized as  the  vertex disjoint union of a clique and a high-girth graph, and that for fixed $m$ most pairs $(m,f)$ cannot be realized as the vertex disjoint union of a clique and a forest.
For some other results concerning sizes of induced subgraphs, see for example Alon and Kostochka \cite{AK}, Alon, Balogh, Kostochka, and Samotij  \cite {ABKS},  Alon,  Krivelevich, and Sudakov \cite{AKS}, Axenovich and Balogh \cite{AB}, Bukh and Sudakov \cite{BS}, Kwan and Sudakov \cite{KS1, KS2} and Narayanan, Sahasrabudhe, and Tomon  \cite{NST}.\\	
 	
 	In this paper we investigate the existence of pairs $(m,f)$ for which we not only have $\sigma(m,f) = 0$, but the stronger property $S_n(m,f) = \emptyset$ for large $n$.

\begin{definition} A pair $(m,f)$ is  \textbf{absolutely avoidable} if there is $n_0$ such that for each $n>n_0$ and for any $e\in \{0, \ldots, \binom{n}{2}\}$, $(n,e) \not\to  (m,f)$.
\end{definition}

Our results show that  there are infinitely many absolutely avoidable pairs.   Our first result gives an explicit construction of infinitely many absolutely avoidable pairs $(m, \binom{m}{2}/2)$. The second one provides an existence result of infinitely many absolutely avoidable pairs $(m, f)$, where $f$ is  \textquotedblleft close\textquotedblright\ to $\binom{m}{2}/2$.
Finally, the last result shows that for every sufficiently large  $m$ congruent to $0$ or $1$ modulo  $4$,  at least one of the pairs 
$(m, \binom{m}{2}/2)$ and $(m, \binom{m}{2}/2-6m)$ is absolutely avoidable.

For the first result  we need to define the following set $M$ of integers. Let 
$$M= \left\{ \frac12 \left( \begin{pmatrix}1 & 0\end{pmatrix}\cdot
									\begin{pmatrix} 3 & 4 \\ 2 & 3 \end{pmatrix}^s\cdot\begin{pmatrix}
										3 \\ 1 \end{pmatrix}+ 5 \right): s\in \mathbb{N} , s\ge 2
									\right\}.$$
In particular, we have $M = \{40, 221,1276\ldots\}$.	\\

\begin{theorem}\label{main}
For any  $m \in M$,   $f = \binom m2 /2$ is an integer and the pair $(m, f)$ is absolutely avoidable. 
\end{theorem}

\vskip 0.1 cm
\begin{theorem}\label{thm:fixed_constants}
	For any monotone integer valued function  $q(m) $ such that $|q(m)|= O(m)$, there are infinitely many values of $m$, such that the pair $(m, \binom{m}{2}/2 -q(m))$  is absolutely avoidable.
	
	Moreover, there are infinitely many values of $m$,  such that for any integer $f'\in (\binom{m}{2}/2-0.175m,\binom{m}{2}/2+0.175m)$ the pair $(m,f')$ is absolutely avoidable.
\end{theorem}

\vskip 0.1cm

\begin{theorem}\label{thm:all_m_are_bad}
	For any $m \ge 740$ with $m \equiv 0,1 \pmod 4$ either $(m, \binom m2 /2)$ or $(m, \binom m2/2 - 6m)$ is absolutely avoidable.
\end{theorem}

\vskip 0.1cm

The main idea of the proofs is that for certain pairs $(m,f)$, there is no graph on $m$ vertices and $f$ edges which is a vertex disjoint union of a clique and a forest or a complement of a vertex disjoint union of a clique and a forest. In order to do so, we need several number theoretic statements that we prove in several lemmas. After that, we use the observation from \cite{EFRS}, that for any sufficiently large $n$, and any $e \leq c \binom{n}{2}$, for any $0\leq c < 1$, there is a graph on $n$ vertices and $e$ edges that is the vertex disjoint union of a clique and a graph of girth greater than $m$. In particular, any $m$-vertex induced subgraph of such a graph is a disjoint union of a clique and a forest. Considering the complements, we deduce that $(m, f)$ is absolutely avoidable. \\

The problem can also be considered in a bipartite setting. It would be interesting to show whether there are absolutely avoidable pairs. Unfortunately we cannot use our method to find such pairs, since any bipartite pair $(m,f)$ with $f\leq m^2/2$ can be represented as the vertex disjoint union of a complete bipartite graph and a forest, see  \Cref{sec:bip}.\\

We state and prove the lemmas in Section \ref{lemmas} and prove the theorems in Section \ref{proofs-theorems}.

\section{Lemmas and number theoretic results}\label{lemmas}

For a positive real number $x$, let $[x]=\{0, 1, \ldots, \lfloor x \rfloor\}$.  We say that a pair $(m,f)$  is {\it realizable } by a graph $H=(V,E)$ if $|V(H)|=m$ and  $|E(H)|=f$. 
For two integers $x, y$, $x\leq y$, we denote by $[x,y]$ the set of all integers at least $x$ and at most $y$. 
For two reals $x, y$, $x\leq y$, we use the standard  notation $(x, y), [x, y), (x, y],$ and $[x,y]$ for respective intervals of reals. For $x \in \mathbb R$ let $\{x\} = x - \floor{x}$ denote the fractional part of $x$, i.e. $\{x\} \in [0,1)$  and $\{x\} = x \pmod 1$. A  real valued sequence $(x_n)_{n \in \mathbb N}$ is called \emph{uniformly distributed modulo 1} (we write u.d. mod $1$) if for any pair of real numbers $s, t$ with $0 \le s < t \le 1$ we have 
$$ \lim\limits_{N \to \infty} \frac{|\{n: 1 \le n \le N, \{x_n\} \in [s,t)  \} |}{N} = t-s.$$ 
We will use the following facts:

\begin{lemma}\label{fact}
	\begin{enumerate}
		\item[(a)] The sequence $(x_n) =\alpha n$ is u.d.  mod $1$ for any $\alpha \in \mathbb R \setminus \mathbb Q$.
		\item[(b)] If a real valued sequence $(x_n)$ is u.d. mod 1 and a real valued  sequence $(y_n)$  has  the property $\lim\limits_{n \to \infty} (x_n - y_n) = \beta$, a real constant, then  $(y_m)$ is also u.d. mod $1$.
	\end{enumerate}
\end{lemma}
For proofs of these facts see for example  Theorem 1.2  and Example 2.1 in \cite{KN}.\\

The following lemma is given in  \cite{EFRS}, we include it here for completeness.
\begin{lemma}\label{yes-clique-forest}
Let $p \in \mathbb N$ and $c$ be a constant $0\leq c<1$. Then for $n \in \mathbb N$ sufficiently large and any $e \in [c\binom{n}{2}]$,  there exists  a non-negative integer $k$ and a graph on $n$ vertices and $e$ edges which is the  vertex disjoint union of a clique of size $k$ and a graph on $n-k$ vertices of girth at least $p$. \end{lemma}

\begin{proof}
Let $p>0$ be given. We use the fact that for any $v$ large enough there exists a graph of girth $p$ on $v$ vertices with $v^{1+\frac{1}{2p}}$ edges. For a probabilistic proof of this fact see for example Bollob\'as \cite{B}  and for an explicit construction see Lazebnik et al. \cite{LUW}.
Let $n$ be a given sufficiently large integer. Let $e\in  [c\binom{n}{2}]$.
  Let $k$ be a non-negative integer such that $\binom{k}{2}\le e \le \binom{k+1}{2}-1$.  Note that since $e \leq c\binom{n}{2}$, $\binom{k}{2} \leq  c\binom{n}{2}$, thus $k\leq \sqrt{c}n +1\leq c'n$, where $c'$ is a constant, $c'<1$.
We claim that $(n,e)$ could be represented as a vertex disjoint union of a clique on $k$ vertices and a graph of girth at least $p$. 
For that, consider a graph $G'$ on $n-k$ vertices and girth at least $p$ such that $|E(G')| \geq (n-k)^{1+\frac{1}{2p}}$.
Consider $G''$, the vertex disjoint union of $G'$ and $K_k$.  Then $|E(G'')| \geq \binom{k}{2} + (n-k)^{1+\frac{1}{2p}}  \geq \binom{k+1}{2} \geq e$.
Here, the second inequality holds since $(n-k)^{1+ \frac{1}{2p}} \geq k$ for  $k\leq c'n$ and $n$ large enough.
Finally, let $G$ be a subgraph of $G''$ on $e$ edges, obtained from $G''$ by removing some edges of $G'$. Thus, $G$ is the vertex disjoint union of a clique on $k$ vertices and a graph of girth at least $p$. 
\end{proof}

~\\
We shall need two number theoretic lemmas for the proof of the main result. Below the set $M$ is defined as in the introduction.

\begin{lemma}\label{lem:pell}
For any $m \in M$, $m$ is a positive integer congruent to $0$ or $1$ modulo 4, and $\sqrt{2m^2 - 10m  + 9} $ is an odd integer for each $m\in M$.
\end{lemma}
\begin{proof}
Recall that $M= \left\{ \frac12 \left( \begin{pmatrix}1 & 0\end{pmatrix}\cdot
									\begin{pmatrix} 3 & 4 \\ 2 & 3 \end{pmatrix}^s\cdot\begin{pmatrix}
										3 \\ 1 \end{pmatrix}+ 5 \right): s\in \mathbb{N} , s\ge 2
									\right\}.$
We see, that $M$ corresponds to the following recursion:  $(x_0, y_0)=(3,1)$ and for $s\geq 0$
\begin{eqnarray*}
x_{s+1} = 3x_s + 4y_s\\
 y_{s+1} = 2x_s + 3y_s.
 \end{eqnarray*}
 
I.e.,   for $s\geq 0$, $$\begin{pmatrix}x_s\\y_s\end{pmatrix} = \begin{pmatrix} 3 & 4 \\ 2 & 3 \end{pmatrix}^{s}\cdot  \begin{pmatrix}3\\ 1\end{pmatrix}.$$
Indeed, $M=\{(x_s+5)/2 : s\geq 2\}$.\\

From the recursion we see that $x_{2s} \equiv3 \pmod 8 $,  $x_{2s+1}\equiv5\pmod 8$, $y_{4s} = y_{4s+1} \equiv 1 \pmod 8$, and  $y_{4s+2} = y_{4s+3} \equiv 5 \pmod 8$ for $s \in \mathbb N_{0}$.   In particular $y_s$ is an odd integer.  Let $m_s= (x_s+5)/2$, i.e., $M= \{m_s: ~ s\geq 2\}.$
When $s$ is even, $m_s \equiv 0 \pmod 4$, and if $s$ is odd,  $m_s \equiv 1 \pmod 4$. This proves the first statement of the Lemma.\\

Next, we observe that $(x,y)=(x_s,y_s)$ gives an integer solution to the generalized Pell's equation 
	\begin{equation}
		 x^2 - 2y^2 = 7. \tag{$*$}
	\end{equation}
Indeed, $(x, y) = (x_0, y_0)=(3,1)$  satisfies $(*)$.
Assume that   $(x,y)= (x_s, y_s)$ satisfies $(*)$.  Let $(x, y)=(x_{s+1}, y_{s+1})$ and insert it 
 into the left hand side of $(*)$.  Then we have  $$x_{s+1}^2 - 2y_{s+1}^2  = 9 x_s^2 + 24 x_sy_s + 16 y_s^2 - 8 x_s^2 - 24 x_sy_s - 18 y_s^2 = x_s^2 -2y_s^2 = 7.$$ Thus $(x, y) = (x_{s+1}, y_{s+1})$  also satisfies $(*)$. \\
 
Since $(x_s, y_s)$ satisfies ($*$), we have that  $y_s = \sqrt {\frac12(x_s^2 -7)} $.  Then $y_s = \sqrt{\frac12((2m_s-5)^2 -7)} = \sqrt{\frac12(4m_s^2 - 20m_s +18)} = \sqrt{2m_s^2 - 10m_s + 9}$. Since $y_s$ is an odd integer, the second statement of the Lemma follows.
\end{proof}
\vskip 0.5cm

For the next lemmas and theorems we will need the following definitions. Let $m, q\in \mathbb Z$,  $m \ge 5+  2\sqrt{|q|}$.  Let 
\begin{equation}\nonumber
\begin{array}{rlrl}
	y_{q}(m) &=\ \frac{\sqrt{2m^2 - 10m - 8q+ 9}}2,  & ~~ z_{q}(m) &=\ \frac{\sqrt{2m^2 - 2m - 8q+ 1}}2, \\
	&&&\\
	t_{q}(m) &=\ z_{q}(m) - y_{q}(m),  &~~ d_{q}(m) &=\  \frac32 - t_{q}(m),\\
		&&&\\
	L_{q}(m)&=\ \floor{\frac52+y_{q}(m)},  &~~ R_{q}(m) &=\ \floor{\frac12 +z_{q}(m)}.\\
\end{array} 
\end{equation}

Note that since  $m \geq 5 + 2\sqrt{|q|}$, we always have $y_q(m), z_q(m) \in \mathbb R$.\\

\begin{lemma}\label{LR}
 Let $q = q(m), m\in \mathbb Z$,   $m \equiv 0,1 \pmod  4$,  $m\geq 5 + 2\sqrt{|q|}$, and $|q(m)|= O(m)$. 
	\begin{enumerate}
			\item[(a)] We have $t_{q}(m) = \tfrac{2\sqrt 2(1-\frac1m)}{\sqrt{1 - \frac{1}{m} + \tfrac{1-8q}{2m^2}} + \sqrt{1 - \tfrac{5}{m} + \tfrac{9-8q}{2m^2}}}.$ In particular, $\lim_{m\to \infty} d_{q}(m) = \frac{3}{2} - \sqrt{2}$.
			\item[(b)] We have  $L_{q}(m) > R_{q}(m)$ if and only if $\{y_{q}(m)\} \in \left[0,d_{q}(m)\right) \cup \left[\frac12, 1\right)$. In particular, $L_{0}(m)>R_{0}(m)$ if  $m\in M$.
	\end{enumerate}
\end{lemma}

\begin{proof}

We start by proving (a). By definition of $t_{q}(m)$ we have
\begin{align*}
	t_{q}(m) &\quad= &z_{q}(m) - y_{q}(m) \\
	&\quad= & \frac12 \sqrt{2m^2 - 2m - 8q + 1} - \frac12\sqrt{2m^2 - 10m - 8q + 9} \\
	&\quad = & \frac12 \frac{2m^2 - 2m - 8q + 1 - 2m^2 + 10m + 8q - 9}{\sqrt{2m^2 - 2m - 8q + 1} + \sqrt{2m^2 - 10m - 8q + 9}}\\
	%&\quad = & \frac{4(m - 1)}{\sqrt 2\, m\left(\sqrt{1 - \frac{1}{m} + \frac{1-8q}{2m^2}} + \sqrt{1 - \frac{5}{m} + \frac{9-8q}{2m^2}}\right) } \\
	&\quad= & \frac{2\sqrt 2(1-\frac1m)}{\sqrt{1 - \frac{1}{m} + \frac{1-8q}{2m^2}} + \sqrt{1 - \frac{5}{m} + \frac{9-8q}{2m^2}}}.
\end{align*}

This also shows that for  $|q| = |q(m)| \in O(m)$,
$\lim_{m\to \infty} d_{q}(m) =    \frac{3}{2}  - \lim\limits_{m \to \infty}t_{q}(m) =  \frac{3}{2} - \sqrt{2}$, which concludes the proof of  (a).\\

Now we can prove part (b). 
From part (a) we have in particular that  $t_{q}(m) = \sqrt 2 + \epsilon_{q}(m)$, where for $m$ sufficiently large $|\epsilon_{q}(m)| < 0.05$, and thus, $t_{q}(m) \in (1, \frac32)$. Thus, $d_{q}(m) = \frac32 -t_{q}(m) \in (0,\frac12)$ for sufficiently large $m$. We compare $L_{q}(m)$ and $R_{q}(m)$ using the expression $x = \lfloor x \rfloor  + \{x\}$:
\begin{eqnarray*} 
L_{q}(m) &= &  \!\floor{\frac52 + y_{q}(m)} \\
&= & 2 + \floor{y_{q}(m)}  + \floor{\frac 12 + \{y_{q}(m)\}} \\ 
& = & 2 + \floor{y_{q}(m)} + \begin{cases}
	0, & \{y_{q}(m)\} \in [0,\frac12) \\ 1, & \{y_{q}(m)\} \in [\frac12,1)
\end{cases},
\end{eqnarray*}

\begin{eqnarray*} 
R_{q}(m)  & =  & \floor{\frac12 + z_{q}(m)} \\
& = &  \floor{\frac12 + y_{q}(m) + t_{q}(m)}\\
& = & \floor{y_{q}(m)}  + \floor{\frac12 + t_{q}(m) + \{y_{q}(m)\}} \\
& = &  \floor{y_{q}(m)}  + \begin{cases}
	1, & t_{q}(m) + \{y_{q}(m)\} \in [1,\frac32) \\
	2, & t_{q}(m) + \{y_{q}(m)\} \in [\frac32,\frac52)
\end{cases}.
\end{eqnarray*} 
Thus 
\begin{align*}
L_{q}(m) - R_{q}(m) &= 2 + \begin{cases}
	0 -1, &  \{y_{q}(m)\} \in [0,\frac12) \text{ and } t_{q}(m) + \{y_{q}(m)\} \in [1,\frac32) \\ 
	0 - 2, &  \{y_{q}(m)\} \in [0,\frac12) \text{ and } t_{q}(m) + \{y_{q}(m)\} \in [\frac32,\frac52)\\
	1-1, &  \{y_{q}(m)\} \in [\frac12,1) \text{ and } t_{q}(m) + \{y_{q}(m)\} \in [1,\frac32)\\
	1-2, &   \{y_{q}(m)\} \in [\frac12,1) \text{ and } t_{q}(m) + \{y_{q}(m)\} \in [\frac32,\frac52)
\end{cases}. 
\end{align*}
So, $L_{q}(m) - R_{q}(m) > 0 $  in all cases except for the second one, i.e.,  if and only if 
\begin{eqnarray*}
 \{y_{q}(m)\} & \in  &[0,1) \setminus \left(\left[0,\tfrac12\right)\cap  \left[\tfrac32- t_{q}(m) ,\tfrac52-t_{q}(m) \right)\right)\\
& =  & \left[\tfrac12, 1\right)  \cup \left([0,1) \setminus \left[d_{q}(m), 1+d_{q}(m)\right)\right)\\
& =  & \left[\tfrac12, 1\right)  \cup [0,d_{q}(m)).
\end{eqnarray*}

Now  let $m \in M$ and consider $ y_{0}(m) = \frac{\sqrt{2m^2-10m + 9}}{2}$. Then by \Cref{lem:pell}, $2y_{0}(m)$ is an odd integer for all $m \in M$, i.e. $\{y_{0}(m)\} = \frac12$. Thus, we have $L_{0}(m) > R_{0}(m)$ for all $m \in M$, which concludes the proof of (b).
\end{proof}

\vskip 0.5cm

\begin{lemma}\label{no-clique-forest-general}
If  $q=q(m) \in \mathbb Z$,  $m \in \mathbb N$, $m \equiv 0,1 \pmod 4$, $m \geq 2 \sqrt{|q|} + 5$,   and  $L_{q}(m) > R_{q}(m)$, then the pair $(m, \binom m2/2 - q)$ cannot be realized as the vertex disjoint union of a clique and a forest.
\end{lemma}
\begin{proof}
	Let $f= \binom{m}{2}/2 - q$. Suppose that $(m,f)$ can be realized as the vertex disjoint union of a clique $K$ on $x$ vertices and a forest $F$ on $m-x$ vertices. We shall show that $L_{q}(m) \le R_{q}(m)$. \\

	\textbf{Claim 1:} $ x\ge L_{q}(m)$.\\
The forest $F$  has $f- \binom{x}{2}= \binom{m}{2}/2 - q- \binom{x}{2} $ edges.  Since $F$ has $m-x$ vertices, it contains  strictly less than $m-x$ edges. 
	Thus, we have  $\binom  m2 /2 - q - \binom x2  < m-x. $ 
Solving for $x$ gives
	$$	x > \frac32 + \frac{1}{2} \sqrt{2m^2 - 10m -8q  + 9 } ~~ \mbox{ or }  ~~x < \frac32 - \frac{1}{2} \sqrt{2m^2 - 10m -8q + 9 }.$$
	 Since $m \geq 2 \sqrt{|q|} + 5$, we have $2m^2 - 10m -8q +9 \ge 9$.  The second inequality gives $x< \frac32 -\frac12\sqrt{2m^2-10m -8q+9}$, and thus  $x < 0$, a contradiction. So only the first inequality for $x$ holds and implies that 
	$$	x \ge \floor{\frac{3 + \sqrt{2m^2 - 10m -8q  + 9}}{2}} +1 = L_{q}(m), $$
	which proves Claim 1.\\
	
	\textbf{Claim 2:} $x \le R_{q}(m)$.\\
		The number of edges in the clique $K$ is  at most $f$ and exactly $\binom{x}{2}$. Thus
		$\binom x2  \le f = \binom m2 /2  - q, $
	which implies that 
		$2x(x-1)\le m(m-1) - 4q$. This in turn gives 
		$$x \le  \floor{\frac{1 + \sqrt{2m^2 - 2m -8q+ 1}}{2}} = R_{q}(m),$$
	and  proves Claim 2.\\
Claims 1 and 2 imply that $L_{q}(m)  \le R_{q}(m)$.
\end{proof}

\begin{lemma}\label{lem:inequality-gives-avoidable}
 Let  $q=q(m) \in \mathbb Z$,  $m \in \mathbb N$, $m \equiv 0,1 \pmod 4$, $m \geq 2 \sqrt{|q|} + 5$. If  both $L_{q}(m) > R_{q}(m)$ and $L_{-q}(m) > R_{-q}(m)$, then the pair $(m,f) = (m, \binom m2 /2 - q)$ is absolutely avoidable.
\end{lemma}
\begin{proof}
	Let $m$ satisfy the condition of the Lemma and let $f_- = \binom{m}{2}/2 - q$ and 
	$f_+ = \binom m2/2 + q$. Then by  \Cref{no-clique-forest-general}, neither $(m,f_+)$ nor $(m, f_-)$ can be represented as the vertex disjoint union of  a clique and a forest. \\

	By \Cref{yes-clique-forest}, for every sufficiently large $n$, and all $e \le \ceil{\binom n2/2}$  we can realize $(n,e)$ as the vertex disjoint union of a clique and a graph of girth  greater than $m$.  Thus, for each $e\in \{0, 1, \ldots, \binom{n}{2}\}$ there is a graph $G$ on $n$ vertices and $e$ edges such that either $G$ or the complement  $\overline{G}$ of $G$ is a vertex disjoint union of a clique and a graph of girth greater than $m$.\\
	
	If $G$ is the vertex disjoint union of a clique and a graph of girth greater than $m$, then any $m$-vertex induced subgraph of $G$ is a vertex disjoint union of a clique and a forest.  Since $(m,f_-)$ can not be represented as a clique and a forest,  we have $G\not \to (m,f_-)$.
	If  $\overline{G}$ is the vertex disjoint union of a clique and a graph of girth greater than $m$, then as above $\overline{G} \not \to (m,  f_+)$.
	Since $f_- = \binom{m}{2}-f_+$,  we have that $G\not\to (m,f_-)$.
	Thus, $(m,f_-)$ is absolutely avoidable. 
\end{proof}
	
%%%%%%%%%%%%%%%%%%%%%%%%%%%%%%%%%%%%%%%%%%%	
\section{Proofs of the Main Theorems}	\label{proofs-theorems}
%%%%%%%%%%%%%%%%%%%%%%%%%%%%%%%%%%%%%%%%%%%	

\begin{proof}[Proof of Theorem \ref{main}]
Let $m\in M$. By \Cref{lem:pell} we have $m \equiv 0,1 \pmod 4$, so $f = \binom{m}{2}/2$ is an integer.  By \Cref{LR}(b) we have $L_0(m) > R_0(m)$. Now we can apply \Cref{lem:inequality-gives-avoidable} with $q=0$. 
Thus, the pair $(m,f)$ is absolutely avoidable.
\end{proof}

 \vskip 0.5cm 

\begin{proof}[Proof of \Cref{thm:fixed_constants}]
	Let $q = q(m) \in \mathbb Z, ~~|q(m)| \in O(m)$, be a monotone function.  	Recall that 	$y_{q}(m) = \frac12 \sqrt{2m^2 - 10m + 9 - 8q}$.  
Let $ a = \lim\limits_{m \to \infty} \frac{q(m)}{m}.$ \\

{\bf Claim 1:} $\lim\limits_{m \to \infty}\left( \frac{m}{\sqrt 2} - y_{q}(m) \right)=  \frac{5}{2\sqrt 2} + \sqrt{2}a ~~ \mbox{ and }~~ \lim\limits_{m \to \infty}\left( \frac{m}{\sqrt 2} - y_{-q}(m) \right)=  \frac{5}{2\sqrt 2} - \sqrt{2} a. $\\

Observe  that 
	\begin{align*}
		\lim\limits_{m \to \infty}	\left(\frac{m}{\sqrt 2} - y_{q}(m)\right) &=   \lim\limits_{m \to \infty} \frac{m}{\sqrt 2}\left(1 - \sqrt{1 - \frac{5}{m} + \frac{9-8q}{2m^2}} \right) \\ 
		&=   \lim\limits_{m \to \infty}   \frac{m}{\sqrt 2}    \frac{ \frac{5}{m} -\frac{9-8q}{2m^2}}{1 + \sqrt{1 + \frac{5}{m} + \frac{9-8q}{2m^2}}}\\
		&=   \frac{5}{2\sqrt 2} + \lim\limits_{m \to \infty}\frac{\sqrt 2 q}{m}\\
		&= \frac{5}{2\sqrt2} + \sqrt 2 a.
		\end{align*}
Doing a similar calculation for $y_{-q}(m)$ proves  Claim 1.\\~\\

{\bf Claim 2:}  $y_{q}(4m)$  and $y_{-q}(4m)$  are  u.d. mod $1$, and in particular, $y_{0}(4m)$ is u.d. mod $1$.\\

Since $\frac1{\sqrt 2} \in \mathbb R \setminus \mathbb Q$, by \Cref{fact}(a) the sequence $(x_{4m}) = (4m)/\sqrt 2$ is u.d. mod $1$. Since we have $\lim\limits_{m \to \infty } (x_{4m} - y_{q}(4m)) = \frac{5+2\sqrt2a}{2\sqrt2} \in \mathbb R$ and 
$\lim\limits_{m \to \infty } ( x_{4m}- y_{-q}(4m)) = \frac{5-2\sqrt2a}{2\sqrt2} \in \mathbb R$,  by \Cref{fact}(b)  $(y_{q}(4m))$  and $(y_{-q}(4m))$  are  also u.d. mod $1$. This proves Claim 2. \\~\\

	Now, to prove the first part of the theorem, from \Cref{lem:inequality-gives-avoidable} it suffices to find infinitely many integers $m$ such that for $q=q(m)$,   $L_{q}(m) > R_{q}(m)$ and $L_{-q}(m) > R_{-q}(m)$. \\
	
	 By Lemma \ref{LR}(a), we have that $\lim_{ m\to \infty} d_{q}(m) = \lim_{ m\to \infty} d_{-q}(m) = 3/2- \sqrt{2}$. 
	Let $m_0$ be large enough so that  for any $m\geq m_0$, $d_{q}(m)$ and $d_{-q}(m)$ are close to these limits, i.e.,  $|d_{q}(m) - (3/2 - \sqrt{2})| < (3/2 - \sqrt{2})/3$ and $|d_{-q}(m) - (3/2 - \sqrt{2})| < (3/2 - \sqrt{2})/3$.\\
	
	 Let $\delta >0$ be a small constant such that $\delta< (3/2 - \sqrt{2})/2$,  $2\delta < 1 - \{ \sqrt{2} a \}$  and  if $\{\sqrt{2} a\} <1/2$, then $\delta < 1/2 - \{\sqrt {2} a\} $.
	 In addition assume that $\delta$ is sufficiently small that for any  $m\geq m_0$,     $\delta< d_{q}(m)/3$, and $\delta < d_{-q}(m)/3$. 
	 Using Claim 1,   define $m_\delta$ to be sufficiently large, so that $m_\delta>m_0$ and for any $m\geq m_\delta$,  $y_{q}(m) - \frac{m}{\sqrt{2}}$  and  $y_{-q}(m) -\frac{m}{\sqrt{2}} $ are  $\delta$-close to the limiting values:

	\begin{align*}
y_{q}(m) &\in \left(\left(\frac{m}{\sqrt{2}} - \frac{5}{2\sqrt{2}} - {\sqrt{2}a}\right) - \delta, \left(\frac{m}{\sqrt{2}} - \frac{5}{2\sqrt{2}} - {\sqrt{2}a}\right)+ \delta\right) \quad \text{ and} \\
	y_{-q}(m) &\in \left(\left(\frac{m}{\sqrt{2}} - \frac{5}{2\sqrt{2}}  +{\sqrt{2}a}\right) - \delta, \left(\frac{m}{\sqrt{2}} - \frac{5}{2\sqrt{2}}  +{\sqrt{2}a}\right)+ \delta \right).
	\end{align*}
	
	We distinguish 2 cases based on the values of $a$:\\

	\textbf{Case 1:} $\{\sqrt2 a\} \in [0, \frac14)\cup [\frac12,\frac34)$, i.e. $\{2\sqrt2 a\} \in [0,\frac12)$.\\

	Since $ \frac {4m}{\sqrt 2} - \frac{5}{2\sqrt 2}$ is a sequence  u.d. mod $1$, there is an infinite set $M_1$ of integers at least $m_\delta$, such that for any $m\in M_1$ 
		$$ \frac{4m}{\sqrt 2} - \frac5{2\sqrt 2} \in (k_{m}+ 1/2 + \{\sqrt2 a\} +\delta, k_{m}+  1/2 + \{\sqrt 2 a\} + 2\delta),$$ for some integer $k_{m}$.  Then we have 
	$$y_{q}(4m) \in \left((1/2+ k_{m}+ \{\sqrt2 a\} +\delta) - {\sqrt{2}a} - \delta, (1/2+ k_{m}+ \{\sqrt2 a\} +\delta) - {\sqrt{2}a}+ \delta\right),$$
	$$y_{-q}(4m) \in \left((1/2+ k_{m}+ \{\sqrt2 a\} +\delta) +{\sqrt{2}a} - \delta, (1/2+k_{m}+ \{\sqrt2 a\} +\delta) +{\sqrt{2}a}+ \delta \right).$$
	This implies that 
	$$\{y_{q}(4m) \}, \{y_{-q}(4m)\} \in \left[1/2, 1\right).$$
From Lemma \ref{LR}(b),  $L_{q}(4m) > R_{q}(4m)$ and $L_{-q}(4m) > R_{-q}(4m)$.  Note that   $f = \binom {4m}2/2- q(4m)$ is an integer. Thus, by \Cref{lem:inequality-gives-avoidable}  the pair $\left(4m, \binom {4m}2/2- q(4m)\right)$ is absolutely avoidable for any $m\in M_1$.
\\

	\textbf{Case 2:} $\{\sqrt2 a\} \in [\frac14, \frac12)\cup [\frac34,1)$, i.e. $\{2\sqrt2 a\} \in [\frac12,1)$.\\

	Since $\frac {4m}{\sqrt 2} - \frac{5}{2\sqrt 2}$ is a sequence  is u.d. mod $1$, there is an  infinite set $M_2$ of integers  at least $m_\delta$, such that for any $m\in M_2$ 

$$\frac{4m}{\sqrt 2} - \frac5{2\sqrt 2} \in (k_{m} + \{\sqrt2 a\} +\delta, k_{m} + \{\sqrt 2 a\} + 2\delta),$$ for some integer $k_{m}$. Then we have 
		\begin{align*}
	y_{q}(4m) &\in \left((k_{m}+ \{\sqrt2 a\} +\delta) - {\sqrt{2}a} - \delta, ( k_{m}+ \{\sqrt2 a\} +\delta) - {\sqrt{2}a}+ \delta\right) \quad \text{ and} \\
	y_{-q}(4m) &\in \left(( k_{m}+ \{\sqrt2 a\} +\delta) +{\sqrt{2}a} - \delta, (k_{m}+ \{\sqrt2 a\} +\delta) +{\sqrt{2}a}+ \delta \right).
			\end{align*} 
	This implies that 
	$$\{y_{q}(4m) \} \in \left[0, 2\delta \right), \{y_{-q}(4m)\} \in \left[1/2, 1\right).$$\\[-0.5cm]

	Recall that  for any $m>m_\delta$, $ \delta< d_{q} (m)/3$. Thus, $\{y_{-q}(4m) \} \in [1/2, 1)$ and 
	$\{y_{q}(4m) \} \in [1/2, 1) \cup \left[0, d_{q}(4m)\right)$. 
From Lemma \ref{LR}(b),  $L_{q}(4m) > R_{q}(4m)$ and $L_{-q}(4m) > R_{-q}(4m)$. Note that  $f = \binom {4m}2/2- q(4m)$ is an integer. Thus, by \Cref{lem:inequality-gives-avoidable} the pair  $\left(4m, \binom {4m}2/2- q(4m)\right)$ is absolutely avoidable for any $m\in M_2$. \\

  This proves the first part of the theorem. \\[0.5cm]

For the second part, let $c= 0.175  < \frac1{4\sqrt 2}$. We shall show  that there is an infinite set $M_0$ of integers such that for any $m\in M_0$ and  for all integers $q \in (-cm ,cm)$,  the pair $(m, \binom{m}{2}/2 - q)$ is absolutely avoidable.  In order to do that, we  shall show that $y_{0}(m)$ does not differ much from $y_{q}(m)$, for chosen values of $m$.\\

Recall that $\lim_{m \rightarrow \infty} d_{q}(m)= 3/2-\sqrt{2} >0$  for any $q\in (-cm, cm)$.  Thus, the interval  $\left[\frac34,\frac34 + d_{q}(m)\right)$ has positive length for any such $q$ and sufficiently large $m$.  By Claim 2 the sequence $y_{0}(4m)$ is u.d. mod $1$, thus there are infinitely many values of $m$ that 
$m\equiv 0 \pmod 4$ and  $\{y_{0}(m)\} \in \left[\frac34,\frac34 + d_{q}(m)\right)$. 
 Now our choice for $m$ will allow us to use Lemmas  \ref{LR}, \ref{no-clique-forest-general} and \ref{lem:inequality-gives-avoidable}.\\
 
 Let  $q \in (-cm, cm)$. It will be easier for us to deal with $y_{q}(m) - y_{0}(m)$ instead of $y_{q}(m)$. Let $s_q(m) = y_{q}(m) - y_{0}(m) $.
 We have  
 \begin{eqnarray*}
 	 \lim\limits_{m \to \infty} s_q(m) &=&  \lim\limits_{m \to \infty} \left(y_{q}(m) - y_{0}(m)\right) \\ 
 	   &= &\lim\limits_{m \to \infty } \frac12 \left(\sqrt{2m^2 - 10m +9-8q} - {\sqrt{2m^2 - 10m +9}}\right) \\
 	   &=& -\sqrt2\lim\limits_{m \to \infty }\frac qm. 
 \end{eqnarray*}

Thus, since $q \in (-cm, cm)$, $c = 0.175 < \frac{1}{4\sqrt2}$, for $m$ sufficiently large we have 
$s_q(m) \in \left(-\frac14, \frac14\right)$.  Since $y_{q}=   s_q(m) + y_{0}(m)$,  and $\{y_{0}(m)\} \in \left[\frac34, \frac34 + d_{q}(m)\right)$,
we have that $\{y_{q}\} = \{s_q(m) + y_{0}(m)\} \in [0, d_{q}(m)) \cup [\frac12,1)$.
\Cref{LR}(b) implies that $L_{q}(m) > R_{q}(m)$ and $L_{-q}(m) > R_{-q}(m)$.
Lemmas  \ref{no-clique-forest-general} and \ref{lem:inequality-gives-avoidable}  then imply that $(m, \binom{m}{2}/2 - q)$ is absolutely avoidable.
\end{proof}

 \vskip 1cm 

%%%%%%%%%%%%%%%%%%%%%%%%%%%%%%%%%%%%%%%%%%%%%
%%%%%%%%%%%%%%%%%%%%%%%%%%%%%%%%%%%%%%%%%%%%%
%%%%%%%%%%%%%%%%%%%%%%%%%%%%%%%%%%%%%%%%%%%%%

\begin{proof}[Proof of \Cref{thm:all_m_are_bad}]
Let $m \ge 740$, $m \equiv 0,1 \pmod 4$.
If $L_{0}(m)>R_{0}(m)$,  by \Cref{lem:inequality-gives-avoidable} $(m, \binom{m}{2}/2)$ is absolutely avoidable, so we assume using Lemma \ref{LR}(b)  that $\{y_0(m)\} \in [d_0(m), \frac12)$.\\

We shall first make some observations about $y_{6m}(m)$ and $y_{-6m}(m)$ by comparing them to $y_{0}(m)$. From the definition we have
$$ y_0(m) = \frac12 \sqrt{2m^2 - 10m + 9}, \quad y_{6m}(m)= \frac12 \sqrt{2m^2 - 58m + 9}, \quad y_{-6m}(m) = \frac12 \sqrt{2m^2 + 38m + 9}.$$
Thus 
$$\lim_{m\to \infty} y_0(m) -y_{6m}(m) = 6\sqrt{2} ~~~ \mbox{ and } ~~ \lim_{m\to \infty} y_0(m) -y_{-6m}(m) = -6\sqrt{2}.$$
By \Cref{LR}(a),  
$$\lim_{m\to \infty} t_0(m) = \lim_{m\to \infty} t_{6m}(m)= \lim_{m\to \infty} t_{-6m}(m) = \sqrt{2}.$$

This implies that  
\begin{eqnarray*}
\lim _{m\to \infty} y_0(m) - y_{6m}(m) - t_{6m}(m) &= &  5\sqrt{2} > 7\\
\lim _{m\to \infty} y_0(m) - y_{6m}(m) +t_{0}(m) & = &  7 \sqrt{2}<10\\
\lim _{m\to \infty} -(y_0(m) - y_{-6m}(m)) +t_{-6m}(m) & = &   7\sqrt{2} <10 \\
\lim _{m\to \infty} -(y_0(m) - y_{-6m}(m)) -t_{0}(m) & = &   5\sqrt{2}> 7.
\end{eqnarray*}

Thus, for sufficiently large $m$ we have 
\begin{eqnarray*}
y_{6m}(m) &< &  y_0(m)  - t_{6m}(m)  -7 \\
y_{6m}(m) & > &y_0 (m)  + t_0(m) - 10\\
y_{-6m}(m) & <  &10  +y_0(m)  - t_{-6m}(m)\\
y_{-6m}(m) & > & 7 +y_0 (m) + t_0(m).
\end{eqnarray*}

Thus, combining these inequalities and recalling that $d_{q}(m) + t_{q}(m)= 3/2$, for any $q$, we have 
$$ y_0(m)  - 8 - \frac12 - d_0(m) < y_{6m}(m)  \le y_0(m) - 8 - \frac12 + d_{6m}(m),$$
$$ y_0(m)  + 8 + \frac12 - d_0(m) < y_{-6m}(m)  \le y_0(m) +8 + \frac12 + d_{-6m}(m).$$

~\\
Recall that  $\{y_0(m)\} \in [d_0(m), \frac12)$.  Recall also that  by \Cref{LR}(a),  $\lim_{m\to \infty} d_q(m) = \frac32-\sqrt 2\approx 0.086$, for $q\in \{0, 6m, -6m\}$.
Then  $\{y_{6m}(m)\} \in [0, d_{6m}(m)) \cup [\frac12, 1)$ and 
 $\{y_{-6m}(m)\} \in [0, d_{-6m}(m)) \cup [\frac12, 1)$. \\
 
This implies by Lemma \ref{LR}(b) that $L_{6m}(m)> R_{6m}(m)$ and $L_{-6m}(m)> R_{-6m}(m)$.
Therefore by Lemma \ref{lem:inequality-gives-avoidable}, the pair $(m, \binom m2/2 -6m)$ is absolutely avoidable.
 
In particular, one can check that all the above inequalities hold for each $m \ge 740$.
\end{proof}

\section{The bipartite setting}\label{sec:bip}
Our entire argument for the existence of absolutely avoidable pairs so far built on the fact that certain pairs $(m,f)$ can not be realized as the disjoint union of a clique and a forest. A similar question can be asked in the bipartite setting: \\

We say a bipartite graph $G$ \emph{bipartite arrows} the pair $(m,f)$, and write $G\overset{bip}{\rightarrow}(m,f)$ if $G$ has an induced subgraph with parts of size $m$ each, contained in the respective parts of $G$, with  exactly $f$ edges. We say that a pair $(n,e)$ of non-negative integers \emph{bipartite arrows} the  pair $(m,f)$, written $(n,e) \overset{bip}{\rightarrow}(m,f)$ if for any bipartite graph $G$  with parts of size $n$ each and with $e$ edges, $G\overset{bip}{\rightarrow}(m,f)$.\\

We call a  pair $(m,f)$ {\it absolutely avoidable in  a bipartite setting}  if there exists $n_0$, such that for each $n \ge n_0$ and for any $e \in \{0,\ldots, n^2\}$, $(n,e) \overset{bip}\nrightarrow (m,f)$. We refer to a complete bipartite graph as a {\it biclique}. We say that a pair $(m,f)$ is {\it bipartite representable} as a graph $H$ if there is a bipartite graph $H$ with $m$ vertices in each part and $f$ edges.
 The following lemma shows that our argument for the existence of such pairs in the non-bipartite case cannot be extended to the bipartite setting. \\

Here, a {\it biclique}  is an induced subgraph of a complete bipartite graph, i.e., could be in particular an empty set or a single vertex.

\begin{lemma}  For any positive integer $m$ and any non-negative integer $f$,  $f \le \floor{\frac{m^2}{2}}$, there is a bipartite graph $H$  with $m$ vertices in each part, $f$ edges, which is the vertex disjoint union of a biclique  and a forest. 
\end{lemma}

\begin{proof}
	
	Fix a pair $(m,f)$ with $f \le \floor{\frac{m^2}2}$. 
	Let $x= \floor{\frac{m}2}$ and let $y$ be the largest integer such that $xy\leq f$. In particular 
	$$xy > f- x ~~~\mbox{ and  }~~~ y \leq \floor{\frac{m^2}2}/ \floor{\frac{m}2}.$$
	We shall use the fact that for any non-negative integers $v'$ and $e'$, with $e'<v'$ and for any partition $v'= v'' + v'''$, with $v'', v'''$ positive integers, 
	there is a forest with partite sets of sizes $v''$ and $v'''$ and $e'$ edges. \\

	{\bf Case 1:} $y <m.$\\
	If $y=0$ then $f< \floor{\frac{m}2}$. In this case $(m,f)$ is bipartite representable as a forest. So, assume that $y>0$.
	We shall show that $(m,f)$ is bipartite representable as a vertex disjoint union of $K_{x,y}$ and a forest.   Let $e' = f - xy$, $v'= 2m -x-y$.  We have that 
	$e' \leq x- 1  = \floor{\frac{m}2}-1$. On the other hand, using the upper bound on $y$, we have that $v' \geq 2m - \floor{\frac{m}2} - \left(\floor{\frac{m^2}2}/ \floor{\frac{m}2}\right).$
	Considering the cases when $m$ is even or odd, one can immediately verify that $e'<v'$.
	Since $x+y +v' = 2m$ and $xy + e'= f$, we have that $(m,f)$ is bipartite representable as a vertex-disjoint union of $K_{x,y}$ and a forest on $v'$ vertices and $e'$ edges.
	Note that in this case we needed $y<m$ so that $K_{x,y}$ doesn't span one of the parts completely. \\

	{\bf Case 2:} $y=m$. \\
	In particular, we have that $ f \geq \floor{\frac{m}2} m$. 
	If $m$ is even, we have that $f\geq m^2/2$ and from our original upper bound $f\leq m^2/2$ it follows that  $f=m^2/2$. Thus $(m,f)$ is bipartite representable as $K_{m/2, m}$ and isolated vertices.
	If $m$ is odd, let $m=2k+1$,  $k\geq 1$.  Then $f\leq  \floor{\frac{m^2}2} = 2k^2 +2k$ and $f \geq y \floor{\frac{m}2} = 2k^2+k$. 
	Consider $K_{k+1, 2k-1}$ and let $e' = f - (k+1)(2k-1)$ and $v'= 2m - 3k$. 
	Then $e' \leq 2k^2 + 2k  - (2k^2 +k -1) = k +1 $ and $v' = 4k+2 - 3k = k+2$. Thus $v'>e'$. 
	Therefore $(m,f)$ is bipartite representable as a vertex disjoint union of $K_{k+1, 2k-1}$ and a forest on $v'$ vertices and $e'$ edges.\\
	
	{\bf Case 3:}. $y=m+1$.\\
	This case could happen only if $m$ is odd. Let $m=2k+1$. Then we have $x=k$ and $y=2k+2$ and $f = 2k^2 + 2k$.  We see that $(m,f)$ is bipartite representable by  $K_{2k, k+1}$ and isolated vertices.
\end{proof}

\section{Conclusion}

We showed that there are infinite sets of absolutely avoidable pairs $(m,f)$.  One could further extend our results and provide more absolutely avoidable pairs.\\

	An analogous to  Theorem \ref{thm:all_m_are_bad}  statement holds for $m \equiv 2,3 \pmod 4$, i.e. for any $m\ge m_0$ either $(m, \floor{\binom m2 /2})$ or $(m, \floor{\binom m2/2} - 6m)$ is absolutely avoidable. We omit the proof here but it can be obtained by a very similar method by slightly changing the constants in the calculations.\\

The arguments in the proof of \Cref{thm:all_m_are_bad} should still hold if we deviate from $f_0=\binom{m}{2}/2$ by a small term, as in \Cref{thm:fixed_constants}. The reason here is that this change does not affect the limit computations for $d_{q}(m)$ and $y_{q}(m)$. Thus,  for each large enough $m$, one should be able to obtain a small interval for $f$ so that  each $(m,f)$ is absolutely avoidable.  We cannot hope to do much better though: In infinitely many cases, if $(m,f_0)$ is absolutely avoidable, then already for $(m, f_0 - m)$ or $(m, f_0 + m)$ our method does not give a contradiction. 
The constant $6$ is  the smallest integer for which the argument in the proof of \Cref{thm:all_m_are_bad} works (since $\{6\sqrt 2\}$ is \emph{close} to $\frac12$ while $\{c\sqrt2\}$, $c \in [5]$ is not). We believe that one could show by an argument very similar to that used in the proof, that for sufficiently large $m$, for any constants $a,b$ which satisfy that $\{a\sqrt 2 - b \sqrt 2\}$ is \emph{close enough} to $\frac 12$, we have that either $(m, f_0 -am)$ or $(m, f_0-bm)$ is absolutely avoidable.\\

As mentioned in Section \ref{sec:bip}, the bipartite setting leaves the following: \\

\textbf{Open Question:} Are there any absolutely avoidable pairs $(m,f)$ in the bipartite setting? \\

\textbf{Acknowledgements:} The authors thank Alex Riasanovsky for his careful reading of the manuscript and his suggestions.


\begin{bibdiv} 
	\begin{biblist} 
		
\bib{AK}{article}{
	title={Induced subgraphs with distinct sizes},
	author={Alon, Noga},
	author={Kostochka, Alexandr},
	journal={Random Structures \& Algorithms},
	volume={34},
	number={1},
	pages={45--53},
	year={2009},
	publisher={Wiley Online Library}
}		
		
		
\bib{ABKS}{article}{
	title={Sizes of induced subgraphs of Ramsey graphs},
	author={Alon, Noga},
	author={Balogh, J{\'o}zsef},
	author={Kostochka, Alexandr},
	author={Samotij, Wojciech},
	journal={Combinatorics, Probability \& Computing},
	volume={18},
	number={4},
	pages={459},
	year={2009}
}

\bib{AKS}{article}{
	title={Induced subgraphs of prescribed size},
	author={Alon, Noga},
	author={Krivelevich, Michael},
	author={Sudakov, Benny},
	journal={Journal of Graph Theory},
	volume={43},
	number={4},
	pages={239--251},
	year={2003},
	publisher={Wiley Online Library}
}

\bib{AB}{article}{
	title={Graphs having small number of sizes on induced k-subgraphs},
	author={Axenovich, Maria},
	author={Balogh, J{\'o}zsef},
	journal={SIAM Journal on Discrete Mathematics},
	volume={21},
	number={1},
	pages={264--272},
	year={2007},
	publisher={SIAM}
}
	
	%% proof of statement for high girth graphs
	\bib{B}{book}{
		title={Extremal graph theory},
		author={Bollob{\'a}s, B{\'e}la},
		year={2004},
		publisher={Courier Corporation}
	}

\bib{BS}{article}{
	title={Induced subgraphs of Ramsey graphs with many distinct degrees},
	author={Bukh, Boris},
	author={Sudakov, Benny},
	journal={Journal of Combinatorial Theory, Series B},
	volume={97},
	number={4},
	pages={612--619},
	year={2007},
	publisher={Elsevier}
}
	
	
		%% Erdos-Furedi original paper
	\bib{EFRS}{article}{
		title={Induced subgraphs of given sizes},
		author={Erd{\H{o}}s, Paul},
		author={F{\"u}redi, Zolt{\'a}n},
		author={Rothschild, Bruce}, 
		author={S{\'o}s, Vera},
		journal={Discrete mathematics},
		volume={200},
		number={1-3},
		pages={61--77},
		year={1999},
		publisher={Elsevier}
	}


	%% He Ma Zhao improvements
\bib{HMZ}{article}{
	title={Improvements on induced subgraphs of given sizes},
	author={He, Jialin}, 
	author={Ma, Jie}, 
	author={Zhao, Lilu},
	journal={arXiv preprint arXiv:2101.03898},
	year={2021}
}


	
%% uniform distribution modulo 1
\bib{KN}{book}{
	title={Uniform distribution of sequences},
	author={Kuipers, Lauwerens},
	author={Niederreiter, Harald},
	year={2012},
	publisher={Courier Corporation}
}


\bib{KS1}{article}{
	title={Ramsey graphs induce subgraphs of quadratically many sizes},
	author={Kwan, Matthew},
	author={Sudakov, Benny},
	journal={International Mathematics Research Notices},
	volume={2020},
	number={6},
	pages={1621--1638},
	year={2020},
	publisher={Oxford University Press}
}


\bib{KS2}{article}{
	title={Proof of a conjecture on induced subgraphs of Ramsey graphs},
	author={Kwan, Matthew},
	author={Sudakov, Benny},
	journal={Transactions of the American Mathematical Society},
	volume={372},
	number={8},
	pages={5571--5594},
	year={2019}
}


	%%explicit construction of dense high girth graphs
	\bib{LUW}{article}{
		title={A new series of dense graphs of high girth},
		author={Lazebnik, Felix},
		author={Ustimenko, Vasiliy},
		author={Woldar, Andrew},
		journal={Bulletin of the American mathematical society},
		volume={32},
		number={1},
		pages={73--79},
		year={1995}
	}

\bib{NST}{article}{
	title={Ramsey graphs induce subgraphs of many different sizes},
	author={Narayanan, Bhargav},
	author={Sahasrabudhe, Julian},
	author={Tomon, Istv{\'a}n},
	journal={Combinatorica},
	volume={39},
	number={1},
	pages={215--237},
	year={2019},
	publisher={Springer}
}

\bib{T}{article}{
	title={On an extremal problem in graph theory},
	author={Tur{\'a}n, Paul},
	journal={Matematikai \'es Fizikai Lapok (in Hungarian)},
	volume={48},
	pages={436--452},
	year={1941}
}



\end{biblist} 
\end{bibdiv}
\end{document}